\documentclass[11pt,oneside,reqno]{amsart}
\usepackage{amssymb,amsmath,amsthm}
\usepackage{a4wide,comment,enumitem,url}
\usepackage{xcolor}
\usepackage{enumitem}
\usepackage{bm}
\usepackage{float}

\usepackage{graphicx}
\usepackage{hyperref}
\usepackage{amsmath,amsopn,amssymb,amsfonts,stmaryrd}
\usepackage{verbatim}
\usepackage{amsthm}
\usepackage{mathtools}
\usepackage{color}
\usepackage{enumitem}
\usepackage[framemethod=TikZ]{mdframed}
\usepackage{bbm}
\usepackage{mathrsfs}
\usepackage{booktabs}
\usepackage{caption}
\usepackage{cancel}
\usepackage{tensor}
\usepackage{cleveref}
\usepackage{ulem}

\newcommand{\sgn}{\mbox{sgn}}

\newcommand{\SL}{\mathrm{SL}}

\newcommand{\N}{\mathbb N}
\newcommand{\C}{\mathbb C}
\newcommand{\Q}{\mathbb Q}

\theoremstyle{plain}
\newtheorem{thm}{Theorem}[section]
\newtheorem{cor}[thm]{Corollary}
\newtheorem{lem}[thm]{Lemma}
\newtheorem{prop}[thm]{Proposition}
\newtheorem{conj}[thm]{Conjecture}

\theoremstyle{definition}

\newtheorem{defn}[thm]{Definition}

\newtheorem{rem}{Remark}

\numberwithin{equation}{section}
\numberwithin{thm}{section}

\setlist[enumerate]{leftmargin=*,label=\rm{(\arabic*)}}

\renewcommand{\sgn}{\textnormal{sgn}}

\renewcommand{\sgn}{{\rm sgn}}
\newcommand{\R}{\mathbb R}
\newcommand{\Z}{\mathbb Z}

\setlist[itemize]{noitemsep, topsep=0pt}

\allowdisplaybreaks

\makeatletter
\newcommand{\vast}{\bBigg@{2}}
\newcommand{\Vast}{\bBigg@{5}}
\makeatother

\renewcommand{\pmod}[1]{\ \left( \mathrm{mod} \, #1 \right)}

\newcommand{\lcm}{\operatorname{lcm}}

\title{Prime-detecting quasi-modular forms in higher level}
\author{Ben Kane}
\address{The University of Hong Kong, Department of Mathematics, Pokfulam, Hong Kong}
\email{bkane@hku.hk}
\author{Krishnarjun Krishnamoorthy}
	\address{Department of Mathematics, National Institute of Technology Tiruchirappalli, Trichy, Tamil Nadu, 620015, India}
        	\email[Krishnarjun Krishnamoorthy]{krishnarjunmaths@outlook.com}
\author{Yuk-Kam Lau}
\address{Weihai Institute for Interdisciplinary Research, Shandong University, China and Department of Mathematics, The University of Hong Kong,  Pokfulam, Hong Kong}
\email{yukkamlau@hku.hk}

\begin{document}

\date{\today}
\keywords{Quasi-modular forms, sign changes of Fourier coefficients}
\subjclass[2020]{11F11,11F30}
\begin{abstract}
	In a previous work, the authors resolved a conjecture about the structure of prime-detecting quasi-modular forms by studying sign changes occurring in quasi-modular cusp forms. In this paper, we extend the considerations to prime-detecting quasi-modular forms of higher level, in particular describing the structure of the space of quasi-modular forms that detect primes in various arithmetic progressions. We also provide an ``analytic'' proof of the level one case.
\end{abstract}
\maketitle

\section{Introduction}

    Recently Craig, van Ittersum and Ono \cite{CvIO}, showed that the set of primes is ``partition theoretic'', meaning that the set of primes can be described as the set of solutions of certain Diophantine equations involving partition functions. In fact they showed that there is an infinite family of such partition-theoretic identities that ``strongly detect'' primes. To describe one of the simplest such examples, given an $a\in \N$, we define the MacMahon partition function
    \begin{equation}\label{eqn:McMahonDef}
    M_a(n) := \underset{n = m_1s_1 + \ldots + m_as_a}{\sum_{0 < s_1 <\ldots < s_a}} m_1\ldots m_a.
    \end{equation}
    Then one of the results of \cite{CvIO} states that an integer $n$ is a prime if and only if
\begin{equation}\label{eqn:PrimeDetectingExample}
    (n^2-3n+2)M_1(n) = 8 M_2(n).
    \end{equation}
    As mentioned before, this is just one of an infinite family of such relations. Some more recent results are available in \cite{Craig, GomezAP, Kang}.

    \begin{defn}\label{defn:Detect}
        A sequence of numbers $a(n)$ is said to \textit{detect} a set $A\subseteq \N$ if $a(n)=0$ whenever $n\in A$. We say that $a(n)$ \textit{strongly detects} $A$ if in addition, $a(n) > 0$ whenever $n\notin A$.
    \end{defn}

    The existence of prime-detecting partition identities arose within the larger context of (mixed-weight) quasi-modular forms whose $n$-th Fourier coefficient detects (or strongly detects) primes.  In particular, define a subset $\Omega$ of the graded ring of (integer weight) quasi-modular forms (of full level) such that $f\in\Omega$ if and only if for ($q:=e^{2\pi i\tau}$)
    \[
    f(\tau)=\sum_{n\geqslant 0} c_f(n) q^n,
    \]
    we have $c_f(n)$ strongly detects the primes. Let $\mathcal{E}$ denote the space of \begin{it}quasi-modular Eisenstein series\end{it} (i.e., the vector space spanned by Eisenstein series and their derivatives). In \cite[Theorem 2.3]{CvIO} the authors classify $\mathcal{E}\cap\Omega$, and propose the following conjecture.
    \begin{conj}\label{Conj:CvIO}
        With notation as above, $\Omega \subset \mathcal{E}$.
    \end{conj}

    In a recent work \cite{KKL1}, we prove this conjecture for quasi-modular forms of full level. In fact we deduce Conjecture \ref{Conj:CvIO} as a consequence of the following slightly stronger result. Let $\widetilde{\Omega}$ be the set of all quasi-modular forms that detect primes. By this we mean that the Fourier coefficients of the quasi-modular form detect primes as in Definition \ref{defn:Detect}.
    \begin{thm}\label{thm:KKL}
        With notation as above, $\widetilde{\Omega} \subset \mathcal{E}$. In particular, $\Omega \subset \mathcal{E}$.
    \end{thm}

    
    There are at least two (somewhat interdependent) natural ways to generalize Theorem \ref{thm:KKL}, which shall be the focus of this paper. The first involves allowing quasi-modular forms of higher level, and the second involves detecting primes in arithmetic progression. Naturally, to detect primes in arithmetic progression, we expect to consider quasi-modular forms of higher level.

Before moving on to a generalization, let us give an example that resembles \eqref{eqn:PrimeDetectingExample}. Emulating \eqref{eqn:McMahonDef}, for $a\in\N$, two  subsets $S_1,S_2\subseteq\{1,2,\dots,a\}$, and a character $\chi$ modulo some $N\in\N$, consider the weighted partition count
\[
M_{a,S_1,S_2,\chi}(n):=\underset{n = m_1s_1 + \ldots + m_as_a}{\sum_{0 < s_1 <\ldots < s_a}} \prod_{j\in S_1} \chi(m_j) \prod_{j\in S_2} m_j,
\]
so that $M_a$ is the case where $S_2=\{1,2,\dots,a\}$ and $\chi$ is trivial. Then it is easy to check that $n\equiv 1\pmod{3}$ is prime if and only if
\begin{equation}\label{eqn:E2Example}
M_1(n)=\frac{n+1}{2} M_{1,\{1\},\emptyset,\chi_{-3}}(n),
\end{equation}
where $\chi_d(n):=\left(\frac{d}{n}\right)$ for a discriminant $d$ is the real Kronecker--Jacobi--Legendre character. The example \eqref{eqn:E2Example} will turn out to arise from one of the forms that we consider in our generalization (see Remark \ref{rem:E2Example}). In order to give the precise generalization of Theorem \ref{thm:KKL}, let us introduce some notation. For $N,M\in\N$ and $m\in\Z$ with $\gcd(m,M)=1$, we let $\widetilde{\Omega}_{m,M}(\Gamma_1(N))$ be the set of quasi-modular forms $f$ of level $N$ (i.e., modular on $\Gamma_1(N)$) for which $c_f(p)=0$ for every prime $p\equiv m\pmod{M}$ with $p\nmid N$. Similarly, we let $\Omega_{m,M}(\Gamma_1(N))$ be those forms which (whose Fourier coefficients) are vanishing only at the primes in the corresponding arithmetic progression; that is $f\in \Omega_{m,M}(\Gamma_1(N))$ if and only if 
    \[
    \left\{n\equiv m\pmod{M}: \gcd(n,N)=1,\ c_{f}(n)=0\right\}=\left\{p\equiv m\pmod{M}: p\text{ prime},\ p\nmid N\right\}.
    \]
    More generally we may define $\widetilde{\Omega}_{m,M}(\Gamma_0(N), \chi)$ to be the set of all quasi-modular forms of level $\Gamma_0(N)$ and character $\chi$ which detect primes etc. We note that $\Omega \subsetneq  \Omega_{0,1}(\SL_2(\Z))$ (if $0\neq f\in \Omega$, then $-f\notin \Omega$ but $-f\in\Omega_{0,1}(\SL_2(\Z))$), as $\Omega_{m,M}(\Gamma_1(N))$ only assumes that $c_f(n)\neq 0$ for $n\equiv m\pmod{M}$ instead of the stronger assumption that $c_f(n)>0$ imposed on $\Omega$. The assumption $c_f(n)>0$ becomes more unnatural for higher level because the Fourier coefficients of Eisenstein series alternate, albeit much more regularly than those of cusp forms (compare Lemma \ref{lem:EisensteinArithmetic} with Lemma \ref{lem:signchangesarithmetic}).

    Naively, we may ask if an analogue to Theorem \ref{thm:KKL} still holds true if we allow forms of higher level. It turns out that this is not true. We describe below some obstructions.
    
    To describe the first obstruction, let $F$ be the prime-detecting quasi-modular form $f_{1,\ell}$ (of level 1) defined in \cite[Lemma 9]{CvIO}  where $\ell\geqslant 3$. Its Fourier coefficients $c_F(n)$ are all non-negative and for $n>1$, $c_F(n)=0$ if and only if $n$ is a prime. Indeed, from the proof, one can see that $c_F(n) \geqslant n^\ell$ for all composite $n>1$.  Suppose $\ell \geqslant 12$ and $N>1$. For any cusp form $f$ of weight $\ell$ and level $1$, $f(Nz)$ is a cusp form of level $N$ and whose coefficients vanish at all primes except possibly $N$ (if $N$ is prime) and have their magnitude bounded by $n^{(\ell-1)/2}d(n)$ (where $d(n)$ is the number of prime divisors of $n$). Fix such a form $f$ with real coefficients and choose a sufficiently large constant $C_{f}$. The quasi-modular form  
    \begin{equation}\label{eqn:NaiveGeneralization}
    E := C_{f} F(z)+ f(Nz)-\frac1{2\pi iN} f'(Nz)
    \end{equation}
    is of level $N$ and strongly detects primes but it does not lie in the space of Eisenstein series.
    
    Furthermore, if $E\in \mathcal{E}(\Gamma_0(N),\chi)\cap \widetilde{\Omega}_{1,3}(\Gamma_0(N),\chi)$, then for any quasi-modular form $f$ of level $N$ and character $\chi$, so is 
    \[
    E+f-f\otimes \chi_{-3},
    \]
    where $f\otimes \chi$ denotes the quadratic twist $f\otimes\chi(\tau)=\sum_{n\geq 0} \chi(n)c_f(n)q^n$ of $f$ and $\chi_D(n):=\left(\frac{D}{n}\right)$ is the Kronecker character. Here $\mathcal{E}(\Gamma)$ denotes the space spanned by Eisenstein series of $\Gamma$ and their derivatives (of arbitrary mixed weights), with $\mathcal{E}(\Gamma_0(N),\chi)$ denoting the quasi-modular forms in $\mathcal{E}(\Gamma_1(N))$ which transform with character $\chi$ on $\Gamma_0(N)$.
    
    Finally, we observe that ``old forms'' also form an obstruction to a generalization of Theorem \ref{thm:KKL} (see Remark \ref{rem3} below). Namely, one can use a construction like \eqref{eqn:NaiveGeneralization} to build higher level non-zero cusp forms which are only supported on $n$ with $\gcd(n,N)>1$. This forces us to consider only those $n$ which are relatively prime to $N$. Hence, instead of considering the question of detection of primes in an arithmetic progression $n\equiv m\pmod{M}$, for level $N$ forms we only consider detection of primes in the arithmetic progressions $n\equiv m\pmod{M}$ with $\gcd(n,N)=1$ (giving a union of arithmetic progressions modulo $\lcm(N,M)$). Although this does indeed cause a restriction on the arithmetic progressions considered, we note that there are only finitely many primes $p\mid N$, so our restriction only excludes finitely many primes.

    Thus, as the above examples show, a naive generalization of Theorem \ref{thm:KKL} is not true. However, in all the obstructions discussed above, we observe that the cuspidal part, albeit non-zero, does not contribute to the Fourier coefficients \textit{within the arithmetic progression being considered} (after restricting to those $n$ relatively prime to the level). Thus, if we restrict our attention to the particular arithmetic progression, the cuspidal part can be considered to be zero.
    
    
    To make matters more precise, we introduce the following sieving operator. For $m\in\Z$ and $M\in\N$, we define the \begin{it}sieving operator\end{it} $S_{M,m}$ acting on quasi-modular forms by 
    \[
    f\big|S_{M,m}(\tau):=\sum_{n\equiv m\pmod{M}} c_f(n)q^n.
    \]
    The sieving operator $S_{M,m}$ maps quasi-modular forms of level $N$ to quasi-modular forms of level $\lcm(N,4M^2,MN)$ (for example, see \cite[Lemma 2.2 (2)]{Vanishing}). 
    Moreover, for $A\subset \Z$, we write $S_A$ for the restriction operator,
    $$
    S_A: \sum_n c(n)q^n\mapsto \sum_{n\in A} c(n)q^n.
    $$
    and set $A_N:=\{n\in \Z: \gcd(n,N)=1\}$.
 
    When detecting primes in the arithmetic progression $m\pmod{M}$, we may obviously add an arbitrary $g|S_{M',m'}$ for any arithmetic progression $m'\pmod{M'}$ that does not intersect the arithmetic progression $m\pmod{M}$. More generally, we may consider all the forms in the kernel of the sieving operator $S_{M,m}$. Observe that, in all the examples above, the cuspidal part was an element of the kernel of the corresponding sieving operator, after restricting to those $n$ relatively prime to the level). Thus $\ker(S_{M,m})$ sits naturally inside $\widetilde{\Omega}_{m,M}(\Gamma_1(N))$ as a subspace. Our main theorem asserts that, once we annihilate the contribution from the kernel of the sieving operator, what remains is Eisenstein.\footnote{One can readily justify excluding elements of the kernel in a reasonable definition of forms that detect primes in an arithmetic progression because \underline{within that arithmetic progression} elements of the kernel do not distinguish primes and non-primes, vanishing at \underline{every} integer in the arithmetic progression.}
        
    \begin{thm}\label{thm:VanishPrimesArithmetic}
    For $M,N\in\N$ and $m\in\Z$ with $\gcd(m,M)=1$, let $\mathcal{V}_{m,M}=\mathcal{V}_{m,M}(N)$ be the subset of $n\in \{0,\dots N-1\}$ for which $(N\Z+n)\cap (M\Z+m)$ is non-empty\footnote{$(N\Z+n )\cap (M\Z+m) \neq \emptyset$ iff $m\equiv n$ mod $(M,N)$. The ``only if" part is clear, while for the ``if" part, we have the solution $x = m-  (m-n) \alpha M/(M,N) = n + (m-n)\beta N/(M,N)$ where $\alpha M+\beta N= (M,N)$.} and $\gcd(n,N)=1$. 
    \begin{enumerate}
        \item Suppose that $f$ is a quasi-modular form of mixed weight and level $N$ and write $f=E+g$ for $E\in\mathcal{E}(\Gamma_1(N))$ and a cuspidal (mixed-weight) quasi-modular cusp form $g$. Then 
        $f\in \widetilde{\Omega}_{m,M}(\Gamma_1(N))$ if and only if $E\in\widetilde{\Omega}_{n,N}(\Gamma_1(N))$ and $g|S_{M,m}|S_{N,n}=0$ for every $n\in\mathcal{V}_{m,M}$. In particular $g|S_{m,M}|S_{A_N}=0$ and
          \begin{equation}\label{eqn:strongInclusion}
            \widetilde{\Omega}_{m,M}\left(\Gamma_1(N)\right)|S_{M,m}|S_{N,n} \subseteq \left(\mathcal{E}\left(\Gamma_1(N)\right)\cap\widetilde{\Omega}_{n,N}\left(\Gamma_1(N)\right)\right)\Big|S_{N,n}\Big|S_{M,m},
        \end{equation} 
        with equality if $M=N$.
    \item
    Moreover, all of the coefficients of elements of $\widetilde{\Omega}_{m,M}(\Gamma_1(N))$ from the arithmetic progression $n\equiv m\pmod{M}$ relatively prime to $N$ come from coefficients of Eisenstein series (and their derivatives) in the sense that 
    \[
    \bigoplus_{n\in\mathcal{V}_{m,M}}\widetilde{\Omega}_{m,M}\left(\Gamma_1(N)\right)|S_{M,m}|S_{N,n}\subseteq\bigoplus_{n\in \mathcal{V}_{m,M}} \left(\mathcal{E}\left(\Gamma_1(N)\right)\cap \widetilde{\Omega}_{n,N}\left(\Gamma_1(N)\right)\right)|S_{M,m}|S_{N,n}. 
    \]
and there exists some $N\mid N'\mid 4N^2$ for which 
\[    \bigoplus_{n\in\mathcal{V}_{m,M}}\widetilde{\Omega}_{m,M}\left(\Gamma_1(N')\right)|S_{M,m}|S_{N,n}\supseteq\bigoplus_{n\in \mathcal{V}_{m,M}(N)} \left(\mathcal{E}\left(\Gamma_1(N)\right)\cap \widetilde{\Omega}_{n,N}\left(\Gamma_1(N)\right)\right)|S_{M,m}|S_{N,n}. 
\]
\item
    In particular, we have 
    \begin{multline*}
\Omega_{m,M}\left(\Gamma_1(N)\right)|S_{M,m}|S_{A_N}\\
=\left(\Bigg(\bigoplus_{n\in \mathcal{V}_{m,M}} \left(\mathcal{E}\left(\Gamma_1(N)\right)\cap \widetilde{\Omega}_{n,N}\left(\Gamma_1(N)\right)\right)|S_{N,n}\Bigg)\cap\left(\Omega_{m,M}\left(\Gamma_1(N)\right)|S_{A_N}\right)\right)\Big| S_{M,m} .
    \end{multline*}
\end{enumerate}
\end{thm}

\begin{rem}
    Before we move forward, we mention that a Galois theoretic proof of Conjecture \ref{Conj:CvIO} was recently obtained in \cite{BvIMOS}. Their proof relies on an extension of the fundamental lemma of Ono-Skinner \cite{Ono-Skinner}. Their proof rests on showing that the Fourier coefficients of cusp forms vary erratically in congruence classes, while the proof of Theorem \ref{thm:KKL} rests on showing that the signs of the Fourier coefficients of cusp forms vary erratically. With regard to forms of higher level, the proof of \cite{BvIMOS} seems to follow through but for an important caveat. The forms considered should not have complex multiplication (in the sense of Ribet \cite{Ribet}). Observe that this is not a restriction for the full level case as there are no non-zero CM forms of full level. Our proof does not require this restriction and works in a fairly general setting.\footnote{We note that this is partially a consequence of our decision to exclude elements of the kernel as ``detecting primes" in Theorem \ref{thm:VanishPrimesArithmetic}; in particular, if a CM form vanishes in an arithmetic progression due to CM properties, then it vanishes for \underline{all} coefficients in that arithmetic progression (not just the primes), so it lies in the kernel component.} We rely on sign changes, a less demanding tool than congruences. The key input here is provided by the prime number theorems\footnote{For an 
    $L$-function $L(s)=\sum_{n\geqslant 1} a_n n^{-s}$ (initially defined for $\Re(s)> 1$), 
    the prime number theorem refers to the asymptotic formula of the form
    \[
    \sum_{p \leqslant x} a_p \;=\; \frac{c\,x}{\log x} \,+\, o\left(\frac{x}{\log(x)}\right),
    \]
    where the constant $c$ may be zero.} for $L$-functions and Rankin-Selberg $L$-functions attached to quasi-modular forms (these are easily seen to be shifts of the corresponding $L$ functions attached to the original holomorphic modular forms and enjoy all of their analytic properties). These results rest on the analytic properties of the associated $L$-functions and therefore ultimately on the theory of newforms, irrespective of whether or not the forms possess complex multiplication.

\end{rem}

    \begin{rem}
        A weaker statement (than Theorem \ref{thm:VanishPrimesArithmetic}) along the lines of Theorem \ref{thm:KKL} reads that, for any $n\in\mathcal{V}_{m,M}$,
        \begin{equation}\label{eqn:WeakInclusion}
            \widetilde{\Omega}_{m,M}\left(\Gamma_1(N)\right)|S_{M,m}|S_{N,n} \subset \mathcal{E}\left(\Gamma_1(N)\right)|S_{M,m}|S_{N,n}.
        \end{equation}
    \end{rem} 
    \begin{rem}\label{rem3}
        The condition $\gcd(n,N)=1$ in the theorem is necessary. There is a natural operator $V_d$ defined by
         \[
        f|V_d(\tau)=f(d\tau)=\sum_{n\geq 0}c_f(n)q^{dn}.
         \]
        This operator sends quasi-modular forms of level $N$ to those of level $Nd$. By applying the operators $V_p$ to forms of level $N'\mid N$, we may artificially force $c_f(p)=0$ for $p\mid N$. For example, if $f$ is a quasi-modular form of level $N$ and $g$ is a quasi-modular form of level $\frac{N}{p}$ with $c_g(1)\neq 0$, then the Fourier coefficients of 
        \[
        f-\frac{c_f(p)}{c_g(1)} g|V_p
        \]
        vanish precisely at the prime $p$ and any $p'$ for which $c_f(p')=0$. 
    \end{rem}

    \begin{rem}

    The sieving operators $S_{M,m}$ and $S_{N,n}$ are a commuting family of projections in the sense that $S_{N,n}\circ S_{M,m} = S_{M,m}\circ S_{N,n} = S_{(n+N\Z)\cap (m+M\Z)}$. 
\end{rem}
Recall that, if non-empty, by the Chinese remainder theorem, $(n+N\Z)\cap (m+M\Z)$ defines an arithmetic progression modulo $\lcm(N,M)$. If $(n+N\Z)$ does not intersect with $(m+M\Z)$, then by $S_{(n+N\Z)\cap (m+M\Z)}$, we denote the zero operator (which annihilates every quasi-modular form). Therefore the direct sum over $\bigoplus_{n\in \mathcal{V}_{m,M}}$ in Theorem \ref{thm:VanishPrimesArithmetic} may be replaced with $\bigoplus_{n\in A_N}$. Writing in this way, we have the following strengthening of Theorem \ref{thm:KKL}.
    \begin{cor}
        When $M=1,m=0$, we have 
        \begin{equation*}            \widetilde{\Omega}_{0,1}\left(\Gamma_1(N)\right)|S_{A_N} \subset \mathcal{E}\left(\Gamma_1(N)\right)|S_{A_N}.
        \end{equation*}   
    \end{cor}

    \begin{rem}\label{rem:kernel} 
    We also note that $S_{N,n}\circ S_{N,n} = S_{N,n}$ (and similarly for $S_{M,m}$ etc.). In particular, $S_{M,m}\circ S_{N,n}$ is a projection into the images of the respective sieving operators.
    Every element in $\widetilde{\Omega}_{m,M}(\Gamma_1(N))$ can be written \textit{uniquely} as a sum of quasi modular forms coming from $\widetilde{\Omega}_{m,M}(\Gamma_1(N))\cap\ker(S_{M,m}\circ S_{N,n})$ and $\widetilde{\Omega}_{m,M}(\Gamma_1(N))|S_{M,m}|S_{N,n}$. Our main theorem now asserts that the second component is Eisenstein.
    
    
    \end{rem}

    As in Remark \ref{rem:kernel}, the space $\widetilde{\Omega}_{m,M}(\Gamma_1(N))$ naturally breaks into two components; the first one arising from the kernel of the sieving operators, and the second arising from quasi-modular Eisenstein series. The sieving operator $S_{m,M}$ is a projection operator and its kernel is quite large as soon as $M \geqslant 2$ (for example given a cusp form $f$, $f-f|S_{m,M}\in \ker(S_{m,M})$). Thus we direct our attention to the Eisenstein series part of $\widetilde{\Omega}_{m,M}(\Gamma_1(N))$. We define the Eisenstein series following Sections 4.5 and 4.6 of \cite{DiamondShurman}. Let $\chi, \psi$ be Dirichlet characters, primitive of conductors $N_1, N_2$ respectively. Let $k\geqslant 2$ be an integer. Suppose that $\chi(-1)\psi(-1) = (-1)^k$. Suppose
    \begin{equation}\label{eqn:EisensteinSeriesDefinition}
        E_{k,\chi,\psi} (\tau) = \delta(\chi) L(1-k, \psi) + 2 \sum_{n=1}^{\infty} \sigma_{k-1}^{\chi,\psi}(n) q^{n},
    \end{equation}
    where as before $q = e^{2\pi i\tau}$ and where
    \begin{equation}\label{eqn:DivisorFunctionDefn}
        \sigma_{k-1}^{\chi,\psi}(n) = \sum_{d | n} \chi\left(\frac{n}{d}\right) \psi(d) d^{k-1}
    \end{equation}
    is the weighted divisor function. The constant $\delta(\psi)$ equals $1$ if $\psi =1$ and is zero otherwise, where we simply write $1$ for the trivial character throughout. If $N$ is the least common multiple of $N_1, N_2$ so that $\chi\psi$ is a primitive Dirichlet character modulo $N$, $E_{k,\chi,\psi}$ is modular of weight $k$ and level $\Gamma_0(N)$ with Nebentypus character $\chi\psi$. For ease of notation, we set 
    \[
    E_{2,1,1}:=-\frac{1}{12}E_2,
    \]
    where $E_2$ is the usual non-holomorphic Eisenstein series of weight $2$, with Fourier expansion given by
    \[
    E_2(\tau) := 1 - 24 \sum_{n=1}^{\infty} \sigma_1(n) q^n=1-24\sum_{n=1}^{\infty} \sigma_{1}^{1,1}(n) q^n.
    \]
For $k=1$, following Section 4.8 of \cite{DiamondShurman}, we also define the weight $1$ Eisenstein series 
\[
E_{1,\chi,\psi}(\tau)= \delta(\chi) L(0, \psi) +\delta(\psi) L(0, \chi)+2 \sum_{n=1}^{\infty} \sigma_{0}^{\chi,\psi}(n) q^{n},
\]
where the $\chi$ and $\psi$ are only taken as an unordered pair in this case, due to the additional symmetry. 

    Ultimately, we want to describe a spanning set for $\widetilde{\Omega}_{m,M}(\Gamma_{1}(N))|S_{A_N}$\footnote{Similar to the $H_{k}$'s defined in \cite{CvIO}.}. In light of the fact that the right-hand side of \eqref{eqn:strongInclusion} only depends on $M$ in the final sieve operator on the right, it is natural to restrict to $M=N$ and find a spanning set of $(\mathcal{E}(\Gamma_1(N))\cap\widetilde{\Omega}_{n,N}(\Gamma_1(N)))|S_{N,n}$ for every $n\in A_N$ (recall that $\mathcal{V}_{m,M}\subseteq A_N$).
    To this end, we make the following definition. Let $N\in \mathbb{N}$ be fixed and choose $n$ coprime to $N$. For every pair of integers $k,\ell \geqslant 1$ with $k+\ell\geq 4$ and a primitive Dirichlet character $\chi$ modulo a divisor of $N$ with $\chi(-1)=(-1)^k$, define
    \[
        H_{k, \ell, \chi,n} := 
    \begin{cases} \overline{\chi(n)}D^{\ell-1}E_{k, 1, \chi} - E_{\ell, 1, \overline{\chi}}&\text{if }k\equiv \ell\pmod{2},\\
    \overline{\chi(n)}D^{\ell-1}E_{k, 1, \chi} - D E_{\ell-1, 1, \overline{\chi}}&\text{if }k\not\equiv \ell\pmod{2},\ \ell\geq 2,\\
    \overline{\chi(n)}E_{k,1,\chi}-\overline{\chi(n)}E_{2,1,\chi}&\text{if }k>2\text{ even, }\ell=1.
 \end{cases}
\]
Here $D$ is the familiar differential operator, defined as
    \[
    D:= \frac{1}{2\pi i} \frac{d}{d\tau},
    \]
To extend the definition of $H_{k,\ell,\chi,n}$ for $2\leq k+\ell\leq 3$, we fix a primitive odd character $\varphi_n$ modulo a divisor of $N$ and assume that $\varphi_n(n)\neq -1$ if such an odd character exists (otherwise $\chi(n)=-1$ for all odd characters). We then define 
\[
C_n:=\begin{cases} 0&\text{if }\varphi_n(n)=-1,\\
\frac{1}{1+\varphi_n(n)}&\text{if }\varphi_n(n)\neq -1.
\end{cases}
\]
Using this, for $k,\ell\in\N$ with $2\leq k+\ell\leq 3$ (and $\chi$ as before) we define    
    \[
    H_{k, \ell, \chi,n} :=
    \begin{cases} 
    \overline{\chi(n)} E_{2,1,\chi}-E_{2,1,1}- (\overline{\chi(n)}-1)C_nE_{1,1,\varphi_n}&\text{if }k=2,\ \ell=1,\chi\neq 1,\\ 
    E_{2,1,1}-C_n \left(DE_{1,1,\varphi_n}+E_{1,1,\varphi_n}\right)&\text{if }k=2,\ell=1,\chi=1,\ \varphi_n(n)\neq -1,\\
DE_{1,1,\chi}-(1+\chi(n)) C_n D E_{1,1,\varphi_n}&\text{if }k=1,\ \ell=2,\\
E_{1,1,\chi} -(1+\chi(n))C_nE_{1,1,\varphi_n}&\text{if }k=\ell=1,
\end{cases}
\]
where we only define $H_{2,1,1,n}$ if $\varphi_n(n)\neq -1$. 
\begin{rem}
Note that, by our convention,   $\varphi_n(n)=-1$ if and only if $\varphi(n)=-1$ for all odd characters, and if all odd characters satisfy $\varphi(n)=-1$, then all even characters satisfy $\varphi(n)=1$ (multiplication by an odd character is a bijection between them and all odd characters satisfy $\varphi(n)=-1$ by assumption). 
\end{rem}
\begin{rem}\label{rem:E2Example}
The difference between the left-hand and right-hand sides of the example \eqref{eqn:E2Example} is simply half the $n$-th Fourier coefficient of $E_{2,1,1}-\frac{1}{2}DE_{1,1,\chi_{-3}}-\frac{1}{2}E_{1,1,\chi_{-3}}$, and for $N=3$ and $n=1$ we have $\varphi_{n}=\chi_{-3}$ and $C_n=\frac{1}{2}$.    
\end{rem}
For $K\geq 4$, we define a vector $\vec{v}$ of length $r=2$ with $j$-th component $(k_j,\ell_j,\chi_j)$ satisfying $k_j+\ell_j=K$ for all $j$. For $2\leq K\leq 3$, we define a vector $\vec{v}$ of length $r=1$ given by $\vec{v}=(k_1,\ell_1,\chi_1)$ satisfying $k_1+\ell_1=K$, with the additional restriction (for $K\leq 3$) that we only allow $\vec{v}=(2,1,1)$ if $\varphi_n(n)\neq -1$. For such $\vec{v}$, we then define 
\[
\mathscr{H}_{K,\vec{v},n}:=\begin{cases}
H_{k_1,\ell_1,\chi_1,n} &\text{if }r=1,\ K\leq 3, \\    H_{k_1,\ell_1,\chi_1,n}-H_{k_2,\ell_2,\chi_2,n}&\text{if }r=2\text{ and }K\geq 4.
\end{cases}
\]

We claim that $\mathscr{H}_{K,\vec{v},n}\in\widetilde{\Omega}_{n,N}(\Gamma_1(N))$ and that these span the space of such forms in the following sense.
\begin{thm}\label{thm:Basis}
        For each $n$ relatively prime to $N\in\N$, the space 
        \[
\widetilde{\Omega}_{n,N}\left(\Gamma_1(N)\right)|S_{N,n}
        \]
        is spanned by (the image under the sieving operator $S_{N,n}$ of)
        \begin{multline*}
            \bigcup_{K=4}^{\infty} \Bigg(\bigg\{\mathscr{H}_{K,\left(\begin{smallmatrix}k_1,\ell_1,\chi_1\\ k_2,\ell_2,\chi_2\end{smallmatrix}\right),n}\ \Big|\ k_j,\ell_j\geq 1,\ k_1+\ell_1 = k_2+\ell_2 = K\bigg\}\\ \cup\bigcup_{K=2}^{3} \Bigg(\bigg\{\mathscr{H}_{K,\left(\begin{smallmatrix}k_1,\ell_1,\chi_1\end{smallmatrix}\right),n}\ \Big|\ k_1,\ell_1\geq 1,\ k_1+\ell_1 = K \ \bigg\}.
        \end{multline*} 
Here $\chi_j$ run through primitive characters modulo divisors of $N$ and for $2\leq K\leq 3$ we have the additional restriction that $(k_1,\ell_1,\chi_1)=(2,1,1)$ only if $\varphi_n(n)\neq -1$ and $(k_1,\ell_1,\chi_1)=(1,\ell_1,\varphi_n)$ with $1\leq \ell_1\leq 2$ only if $\varphi_n(n)=-1$.
 \end{thm}
    \begin{rem}
By \eqref{eqn:strongInclusion}, it is reasonable to consider $\widetilde{\Omega}_{m,M}(\Gamma_1(N))$ for each $n\in\mathcal{V}_{m,M}$ one-by-one, and then write  
\[
f=\sum_{n\in\mathcal{V}_{m,M}} f|S_{N,n}+\sum_{n\notin \mathcal{V}_{m,M}} f|S_{N,n}.
\]
However, note that $f|S_{N,n}$ may have level up to $N^2$, so the forms spanning the projections do not themselves lie in the space of level $N$ forms. Namely, if we take $h_n$ in the above spanning set, then 
\[
\sum_{n\in\mathcal{V}_{m,M}} h_n|S_{N,n}\in \bigoplus_{n\in\mathcal{V}_{m,M}}\Omega_{m,M}(\Gamma_1(N))|S_{N,n}
\]
has some level $N'$ dividing $N^2$, but there may not exist $h$ of level $N$ for which 
\[
h=\sum_{n\in\mathcal{V}_{m,M}} h_n|S_{N,n}+\sum_{n\notin\mathcal{V}_{m,M}} h|S_{N,n}.
\]
So an arbitrary linear combination of forms from the spanning set may not lie in the space we are considering, and we do not consider the problem of determining whether they glue together to give a single form of the correct level.
    \end{rem}
    We also have an analogue of \cite[Theorem 1.3]{KKL1} which may be proved similarly.
    
    \begin{thm}\label{thm:FiniteCheckAP}
        Suppose that $f\in \mathcal{E}(\Gamma_1(N))$. There exists an integer $r$, such that if there exists primes $\{p_1, \ldots, p_r\}$ all congruent to $m$ mod $M$ such that $c_{f}(p_i)=0$ for $i=1,2,\ldots,r$, then $f\in \widetilde{\Omega}_{m,M}(\Gamma_1(N))$.
    \end{thm}
As a corollary of (the proof of) Theorem \ref{thm:FiniteCheckAP} and Theorem \ref{thm:VanishPrimesArithmetic}, one concludes that forms that strongly detect arithmetic progressions in the $N=1$ case actually strongly detect all primes.
    \begin{cor}\label{cor:N=1}
Suppose that $M\in\N$, and $m\in\Z$ with $\gcd(m,M)=1$. If $f\in \widetilde{\Omega}_{m,M}(\SL_2(\Z))$, then $f\in \widetilde{\Omega}$. In particular, level $1$ forms cannot strongly detect primes in an arithmetic progression unless they in fact detect all the primes, that is $M=1$.
    \end{cor}
\begin{rem}
In order to obtain Corollary \ref{cor:N=1}, we combine the techniques involving sign changes from this paper and \cite{KKL1} with the techniques from \cite{BvIMOS} involving $\ell$-adic Galois representations. We use the techniques from \cite{BvIMOS} to establish that certain cusp forms must vanish (see Lemma \ref{lem:SieveVanish}). If one makes the stronger assumption that the coefficient $c_f(n)$ are all nonnegative, then one can show that the desired cusp form vanishes with sign-change techniques, obtaining a weaker version of Corollary \ref{cor:N=1}.
\end{rem}

\begin{rem}
In Theorem \ref{thm:VanishPrimesArithmetic} (1), it is shown that if $f$ detects primes in the arithmetic progression congruent to $m$ modulo $M$, then the Eisenstein series part detects primes in any congruence class modulo $N$ that intersects with the arithmetic progression. However, the cuspidal part is not shown there to vanish in the entire congruence classes modulo $N$ (only in the intersection with the original arithmetic progression). In light of Corollary \ref{cor:N=1}, it is likely that any form that detects all primes congruent to $m$ modulo $M$ also detects all primes in congruence classes $n$ modulo $N$ for $n\in\mathcal{V}_{m,M}$. This requires a more nuanced version of Lemma \ref{lem:SieveVanish} that shows that a level $N$ quasi-modular cusp form $g$ vanishes if $g|S_{M,m}=0$ for some $M$ relatively prime to $N$. The assumption that $M$ is relatively prime to $N$ in this possible extension of Lemma \ref{lem:SieveVanish} is certainly necessary, as CM forms can vanish in arithmetic progressions modulo $N$, for example.
\end{rem}

    We conclude this paper by providing, in an appendix, a purely analytic proof of some of results for the level one case. In particular, we give an equivalent condition (see Theorem \ref{thm:Equivalence} below) in terms of certain ratios of the Riemann zeta function for the linear combination of certain divisor functions to detect primes. This can be used to easily produce examples of such identities which in turn can be translated to prime-detecting identities involving partition functions. As the notation and results leading up to the proof of Theorem \ref{thm:Equivalence} are self-contained and somewhat independent of the rest of the paper, we have given the details in a separate section.
    
    The paper is organized as follows. In Section \ref{sec:signchanges} we show that a quasi-modular cusp form not in the kernel of the sieving operator exhibits infinitely many sign changes at the prime Fourier coefficients. In Section \ref{sec:arithmetic}, we show that the corresponding Eisenstein series exhibit at most finitely many sign changes and conclude the proof of Theorem \ref{thm:VanishPrimesArithmetic}. In Sections \ref{sec:Basis} and \ref{sec:FiniteCheckAP} we prove Theorems \ref{thm:Basis} and \ref{thm:FiniteCheckAP} (together with Corollary \ref{cor:N=1}) respectively.

    \subsection*{Acknowledgements}
    The authors would like to thank Ken Ono and Jan-Willem van Ittersum for their comments on a previous draft of this manuscript. The research was conducted during the conference HKU Number Theory Days 2025. The authors thank the Department of Mathematics at HKU and the Institute of Mathematical Research at HKU for supporting the conference and hosting the second author.
The research of the first author was supported by grants from the Research Grants Council of the Hong Kong SAR, China (project numbers HKU 17314122, HKU 17305923). The third author is supported by GRF (No. 17317822) and NSFC (No. 12271458). The authors would like to thank the anonymous referees for their comments and suggestions.

\section{Sign changes of quasi-modular cusp forms}\label{sec:signchanges}

    We let $\mathcal{S}(\Gamma_1(N))$ denote the space of quasi-modular cusp forms of level $N$ (i.e., the space spanned by cusp forms and their derivatives), and omit $\Gamma_1(N)$ in the notation if $N=1$ (i.e., $\mathcal{S}:=\mathcal{S}(\Gamma_1(1))$).
    
    In this section, we show that Fourier coefficients of quasi-modular cusp forms (of arbitrary level) exhibit sign changes at the primes. First we recall the main sign change lemma (for the full level case) from \cite{KKL1}. 
    
\begin{lem}\label{lem:signchanges}
Suppose that $F\in \mathcal{S}$ has a Fourier expansion 
\[
F(\tau)=\sum_{n\geqslant 1}c_{F}(n)q^n
\]
with $c_{F}(n)\in\R$. If $F\neq 0$, then the sequence $c_{F}(p)$, running over $p$ prime, has infinitely many sign changes. 
\end{lem} 

In higher levels, we have the following lemma.

\begin{lem}\label{lem:signchangesGamma1(N)}
    Suppose that $F\in \mathcal{S}(\Gamma_1(N))$ has a Fourier expansion 
    \[
    F(\tau)=\sum_{n\geqslant 1}c_{F}(n)q^n
    \]
    with $c_{F}(n)\in\R$.  If $F\neq 0$, then either $c_F(n)=0$ for all $n$ satisfying $(n,N)=1$ or $\{c_F(p)\}$ has infinitely many sign changes, as $p$ runs through the primes. 
\end{lem}

\begin{proof}
    Suppose that $F\neq 0$. We may express a quasi-modular cusp form $F$ as the linear combination of Hecke eigenforms and their derivatives;
    \[
    F= \sum_{f} \sum_j A_{f}(j) f^{(j)}
    \]
    where $A_f(j)\in \mathbb{C}$ and $f^{(j)}=D^j f$ is the normalized $j$-th derivative of $f$. We break the sum over $f$ based on the image of the $V$ map. Observe from \cite[Lemma 4]{Li} that it suffices to restrict ourselves to those $d$ which divide $N$. Hence, we write
    \[
    F = \sum_{d|N} G_d
    \]
    where
    \[
    G_d= \sum_{f} \sum_{j} A_f(j) f^{(j)} = \sum_{n=1}^{\infty} c_{G_d}(n) q^{nd}
    \]
    where the $f$ sum is restricted to those Hecke eigenforms in the image of $V_d$ operator.

    From construction, it follows that $c_F(dp)=\sum_{\ell|d} c_{G_\ell}(dp/\ell)$ for almost all primes $p$. In particular, $c_F(p) = c_{G_1}(p)$ for almost all primes $p$ as $c_{G_d}(p) =0$ for every $d>1$ and any prime $p>N$. If $G_1=0$, then  $c_F(n)=0$ for all $(n,N)=1$, because for all $1<d|N$, $G_d$ contributes a coefficient of $0$ to the term $q^n$ whenever $(n,N)=1$. Therefore the lemma would follow if we prove the infinite of sign changes in $c_{G_1}(p)$. This reduces the proof of the lemma to the case of $d=1$.

    When $d=1$, we may, without loss of generality, suppose that $F=G_1$. In this case, the proof is very similar to the proof of Lemma \ref{lem:signchanges}. We briefly sketch the arguments for the sake of completeness and refer the reader to \cite{KKL1} for more details. 

    After some rearrangement, we may write 
    \[
    c_F(p) = \sum_f \sum_j A_{f}(j) p^j c_f(p) = \sum_{f} P_{f}(p) c_f(p)
    \]
    for some polynomials $P_{f}(x)\in \C[x]$. For every $f$, we let the weight of $f$ be $k_{f}\geqslant 1$ 
    and the degree of $P_{f}$ to be $j_{f} \geqslant 0$. Denote the leading coefficient of $P_{f}$ as simply $A_{f}$. From the Ramanujan bound for $a_{f}$ \cite{Deligne1, Deligne2}, we know
    \[
    P_{f}(p)c_f(p) = A_f p^{j_f} c_f(p) + O\left(p^{j_{f} + \frac{k_{f}-1}{2} -1}\right),
    \]
    where the implied constant depends at most on $F$.

    Set $\alpha_0:= \max_{f}\left\{\alpha_f := j_f + \frac{k_f +1}{2}\right\}$. Summing over the primes, and using the prime number theorem in this setting (see \cite[Theorem 5.13]{IwaniecKowalski}) gives us
    \begin{equation}\label{eqm}
    \sum_{p\leqslant x} c_F(p) = o\left(\frac{x^{\alpha_0}}{\log(x)}\right).
    \end{equation}

    Similarly,
    \begin{eqnarray*}
    |c_F(p)|^2 
    &=& \sum_{f,g} P_f(p) \overline{P_g(p)} c_f(p)\overline{c_g(p)} \\
    &=& \sum_f |P_f(p)|^2 |c_f(p)|^2 + \sum_{f\neq g} P_f(p) \overline{P_g(p)} c_f(p)\overline{c_g(p)}.     
    \end{eqnarray*}

    Appealing to Rankin-Selberg theory and the Selberg orthogonality conjecture (a theorem in this case) \cite[Corollary 1.5]{LY}, the sum over $p$ of $|c_F(p)|^2$ is dominated by the diagonal term. In particular the leading order of the asymptotic is obtained from the forms of the largest weight. Plugging this all in we get,
    \begin{align}\label{eqms}
    \sum_{p\leqslant x} |c_F(p)|^2 \gg_{F} \frac{x^{\beta_{0}}}{\log(x)}
    \end{align}
    where $\beta_0 = 2\alpha_0-1$.
    It follows from Deligne's bound that for all prime $p$, 
    \[
    |c_F(p)|\leqslant \sum_{f} |P_{f}(p) c_f(p)| \leqslant \sum_{f} \|P_f\| p^{j_f}|c_f(p)| \leqslant C_F p^{\alpha_0-1}
    \]
    where $\|P\|=\sum_{r=0}^{m} |A_r|$ if $P(x)=\sum_{r=0}^m A_rx^r\in \mathbb{C}[x]$ and $C_F>0$ is a constant.  This yields
    \[
    \sum_{p\leqslant y} |c_F(p)|^2 = O\left(\frac{y^{2\alpha_0-1}}{\log(y)}\right).
    \]
    To complete the proof, suppose, for the sake of contradiction, that there are only finitely many sign changes in the sequence $\{c_{F}(p)\}$. Without loss of generality we may suppose that $c_{F}(p)$ is positive for all $p > y$. Then 
    \[
    \sum_{p\leqslant x} |c_{F}(p)|^2 \leqslant C_{F} x^{\alpha_{0}-1} \sum_{y \leqslant p\leqslant x} c_{F}(p) + O\left(\frac{y^{2\alpha_0-1}}{\log(y)}\right) = o \left(\frac{x^{2\alpha_{0}-1}}{\log(x)}\right).
    \]
    as $x\to\infty$ and with $y$ fixed. This is a contradiction to \eqref{eqms} as $\beta_0 = 2\alpha_0-1$. 
    This completes the proof.
\end{proof}

\begin{cor}\label{lem:signchangesarithmetic}
Let $N\in\N$ be given and suppose that the Fourier expansion 
\[
f(\tau)=\sum_{l\geqslant 1}c_f(l)q^l
\]
of $f\in\mathcal{S}(\Gamma_1(N))$ has real coefficients $c_f(l)\in\R$. If $M\in\N$ and $m\in\Z$ with $\gcd(m,M)=1$, then for any $n\in \Z$ with $\gcd(n,N)=1$, either $f\in \ker(S_{M,m}\circ S_{N,n})$ or $\{c_f(p)\}_{p\equiv m\pmod{M},\, p\equiv n\pmod{N}}$ has infinitely many sign changes.
\end{cor}

\begin{proof}

The form $F:=f|S_{M,m}|S_{N,n}$ is an element of $\mathcal{S}(\Gamma_1(L))$ for some $L|M^4 N^2$. By Lemma \ref{lem:signchangesGamma1(N)}, either $c_F(l)=0$ for all but finitely many  $(l,L)=1$ or $\{c_F(p)\}$ has infinitely many sign changes. As the Fourier coefficients are supported on (a subset of) $(m+M\Z)\cap (n+N\Z)$, all integers in the support are coprime with $L$. The desired result follows.
\end{proof}


\section{Vanishing at primes in arithmetic progressions}\label{sec:arithmetic}

We begin with a lemma similar to \cite[Lemma 4.1]{KKL1}.
\begin{lem}\label{lem:EisensteinArithmetic}
Suppose that $f\in \mathcal{E}(\Gamma_1(N))$ with real Fourier coefficients. Then for every arithmetic progression $n\pmod{N}$ with $\gcd(n,N)=1$ there exists $\varepsilon_{n,N}\in \{-1,0,1\}$ for which every sufficiently large prime $p\equiv n\pmod{N}$ 
\[
\sgn(c_f(p))=\varepsilon_{n,N}.
\]
Moreover, we have $\varepsilon_{n,N}=0$ if and only if $c_f(p)=0$ for all $p\equiv n\pmod{N}$.
\end{lem}
\begin{proof}
The basis elements of $\mathcal{E}(\Gamma_1(N))$ are given by $D^{\ell}E_{k,\chi,\psi}|V_{\delta}$, where $\chi$ and $\psi$ are Dirichlet characters modulo $N$.
Suppose $f\in \mathcal{E}(\Gamma_1(N))$. Let us denote the $j$-th Fourier coefficient of $f$ as $c_{f}(j)$.
For $j=p$ prime, expressing $f$ as a linear combination of the basis above, and using \eqref{eqn:DivisorFunctionDefn}, we see that the $p$-th coefficient can be written as a polynomial 
\[
c_f(p)=\sum_{r} \beta_r(\chi(p),\psi(p)) p^r,
\]
where the coefficients $\beta_r(\chi(p),\psi(p))\in\C$ only depend on $f$, $\chi(p)$ and $\psi(p)$. Since $\chi$ and $\psi$ are characters modulo $N$, for $p\equiv n\pmod{N}$ we have $\chi(p)=\chi(n)$ and $\psi(p) = \psi(n)$, so 
\[
c_f(p)=\sum_{r} \beta_r(\chi(n),\psi(n)) p^r
\]
is a polynomial in $p$ whose coefficients only depend on $f$ and $n$. If this polynomial vanishes identically, then we may take $\varepsilon_{n,N}:=0$, and otherwise we may choose $r_0$ largest so that $\beta_{r_0}(\chi(n),\psi(n))\neq 0$, in which case we may choose 
\[
\varepsilon_{n,N}:=\sgn\left(\beta_{r_0}(\chi(n),\psi(n))\right)\in\{\pm 1\}.
\]
The fact that $\beta_{r_0}(\chi(m),\psi(m))\in \mathbb{R}$ follows from the assumption that the Fourier coefficients of $f$ are real.
\end{proof} 

Theorem \ref{thm:VanishPrimesArithmetic} now follows by an argument similar to the proof of Theorem \ref{thm:KKL}.
\begin{proof}[Proof of Theorem \ref{thm:VanishPrimesArithmetic}]
Suppose that $f\in \widetilde{\Omega}_{m,M}(\Gamma_1(N))$. We split 
\[
f=f_E+f_S
\]
where $f_E \in \mathcal{E}(\Gamma_1(N))$ and $f_{S}\in \mathcal{S}(\Gamma_1(N))$. As in the proof of Theorem \ref{thm:KKL} (see \cite{KKL1}), we may isolate the real and imaginary parts of $f_E, f_S$ and deal with them separately. For brevity, we shall suppose that the Fourier coefficients of $f_E$ and $f_S$ are real valued and move forward.
For any $n\in\mathcal{V}_{m,M}$, Lemma \ref{lem:EisensteinArithmetic} gives us $\varepsilon_{n,N}$ for which 
\[
\sgn\left(c_{f_E}(p)\right)=\varepsilon_{n,N}
\]
for sufficiently large $p\equiv n\pmod{N}$ (and $c_{f_E}(p)=0$ for \underline{all} such $p$ if $\varepsilon_{n,N}=0$). Since $(N\Z+n)\cap (M\Z+m)$ is non-trivial by assumption, there exist infinitely many $p$ in this intersection (since $(N\Z+n)\cap (M\Z+m)$ defines an arithmetic progression and since $(n,N)=(m,M) = 1$), and for such sufficiently large $p$ we have $c_{f_S}(p)=-c_{f_E}(p)$, implying that 
\[
\sgn\left(c_{f_S}(p)\right)=-\sgn\left(c_{f_E}(p)\right)=-\varepsilon_{n,N}.
\]
However, by Corollary~\ref{lem:signchangesarithmetic}, $\{c_{f_S}(p)\}_{p\in (N\Z+n)\cap (M\Z+m)}$ has infinitely many sign changes unless $f_S|S_{M,m}|S_{N,n}=0$.  Therefore, 
\[
\sum_{n\in\mathcal{V}_{m,M}} f_S|S_{M,m}|S_{N,n}=0.
\]


But then 
\[
\sum_{n\in \mathcal{V}_{m,M}}f|S_{M,m}|S_{N,n}=\sum_{n\in\mathcal{V}_{m,M}}f_E|S_{M,m}|S_{N,n}.
\]
(1) Since the vanishing of the coefficients of $f_E|S_{M,m}$ in arithmetic progressions are determined by the vanishing of the $\varepsilon_{n,N}$, and these only depend on $n\pmod{N}$ (or, equivalently, by the polynomials from the proof of Lemma \ref{lem:EisensteinArithmetic} vanishing identically), we see that the $p$-th coefficient of $f_E|S_{M,m}$ vanishes at every $p\equiv m\pmod{M}$ if and only if the $p$-th coefficient of $f_E$ vanishes at every $p\equiv n\pmod{N}$ for every $n\in\mathcal{V}_{m,M}$. This establishes the first statement in (1). Finally, to establish \eqref{eqn:strongInclusion}, suppose that $f\in\widetilde{\Omega}_{m,M}(\Gamma_1(N))$. We use the fact that the cuspidal part is annihilated by $S_{M,m}|S_{N,n}$ (from the first statement in part (1)) to compute 
\[
f|S_{M,m}|S_{N,n} = f_E|S_{M,m}|S_{N,n}.
\]
Since $f_E\in \widetilde{\Omega}_{n,N}(\Gamma_1(N))$ (also by the first statement in part (1)), we conclude that 
\[
f|S_{M,m}|S_{N,n} = f_E|S_{M,m}|S_{N,n}\in \left(\mathcal{E}(\Gamma_1(N))\cap\widetilde{\Omega}_{N,n}(\Gamma_1(N))\right)\Big|S_{M,m}\Big|S_{N,n}.
\]
(2) The first inclusion is immediate from \eqref{eqn:strongInclusion}. To see the second, suppose that $E_n\in \mathcal{E}(\Gamma_1(N))\cap\widetilde{\Omega}_{n,N}(\Gamma_1(N))$ is arbitrary. Then  
\[
E=\sum_{n\in\mathcal{V}_{m,M}} E_n|S_{N,n}
\]
satisfies $a_E(p)=0$ for every $p\equiv m\pmod{M}$ with $\gcd(p,N)=1$. Since $E_n|S_{N,n}$ has level dividing $4N^2$ (arguing as in \cite[Lemma 2.2 (2)]{Vanishing}), the claim follows.
\noindent

\noindent
(3) Since $\Omega_{m,M}(\Gamma_1(N))\subseteq \widetilde{\Omega}_{m,M}(\Gamma_1(N))$, this follows immediately from the first inclusion in part (2), with the reverse inclusion being trivial due to the intersection with $\Omega_{m,M}(\Gamma_1(N))$ on the outside.\qedhere
\end{proof}

\section{Spanning sets for \texorpdfstring{$\widetilde{\Omega}_{n,N}(\Gamma_1(N))|S_{N,n}$}{Tilde-Omega n,N(N) sieved}}\label{sec:Basis}

In this section, we prove Theorem \ref{thm:Basis}. We fix $N\in\N$ and $n\in\Z$ relatively prime to $N$. We first compute the $p$-th coefficient of $H_{k,\ell,\chi,n}$ for $p\equiv n\pmod{N}$.
\begin{lem}\label{lem:Hsingle}
Suppose that $K,N\in\N$, $n\in\Z$ with $\gcd(n,N)=1$, $k,\ell\in\N$ with $k+\ell=K$, and $\chi$ is a primitive character modulo a divisor of $N$. Then for $p\equiv n\pmod{N}$, we have 
\[
\frac{1}{2}c_{H_{k,\ell,\chi,n}}(p)=
\begin{cases}
 p^{K-2}-1   &\text{if $K$ is even},\\
p^{K-2}-p    &\text{if $K$ is odd}.
\end{cases}
\]
\end{lem}
\begin{proof}
For $K>2$ even, we compute
\[
\frac{1}{2}c_{H_{k,\ell,\chi,n}}(p)=\overline{\chi(n)} \left(p^{\ell-1}(\chi(n) p^{k-1} + 1)\right) - (\overline{\chi(n)}p^{\ell-1} + 1) =p^{K-2}-1.
\]
For $k=\ell=1$, we have 
\begin{equation}\label{eqn:weight1}
\frac{1}{2}c_{E_{1,1,\chi}}(p)=1+\chi(p)=1+\chi(n).
\end{equation}
Hence for $\ell=1$ and $k=1$, if $\chi(n)=-1$ then we have $c_{H_{1,1,\chi,n}}(p)=c_{E_{1,1,\chi}}(p)=0$. If $\chi(n)\neq -1$ and $\chi\neq \varphi_n$, then 
\[
\frac{1}{2}c_{H_{1,1,\chi,n}}(p)=(1+\chi(n)) -\frac{1+\chi(n)}{1+\varphi_n(n))}(1+\varphi_n(n))=0.
\]
This completes all of the $K$ even cases.

For $\ell\geq 2$ and $K\geq 3$ odd, we have $H_{k,\ell,\chi,n}=D(H_{k,\ell-1,\chi,n})$, and differentiation multiplies the $p$-th coefficient by $p$, so we have $\frac{1}{2}c_{H_{k,\ell,\chi,n}}(p)=p(p^{K-3}-1)=p^{K-2}-p$.

For $\ell=1$ and $K>3$ odd, we have
\[
\frac{1}{2}c_{H_{K-1,1,\chi,n}}(p)=\overline{\chi(n)} \left(\chi(n) p^{k-1} + 1\right) - \overline{\chi(n)}(\chi(n)p + 1) =p^{K-2}-p.
\]
For $\ell=1$, $k=2$, and $\chi(n)=1$, $\chi\neq 1$, we have 
\[
\frac{1}{2}c_{H_{2,1,\chi,n}}(p)=(p+1)-(p+1)=0.
\]
For $\ell=1$, $k=2$, and $\chi(n)\neq 1$, we use \eqref{eqn:weight1} to compute
\[
\frac{1}{2}c_{H_{2,1,\chi,n}}(p)=\overline{\chi(n)} \left(\chi(n) p + 1\right) - (p+1)-\frac{\overline{\chi(n)}-1}{1+\varphi_n(n)} (1+\varphi_n(n))=0.
\]
\end{proof}
As an immediate corollary to Lemma \ref{lem:Hsingle}, we see that the $p$-th coefficient of $\mathscr{H}_{K,\vec{v},n}\in\widetilde{\Omega}_{n,N}(\Gamma_1(N))$ vanishes in the arithmetic progression.
\begin{lem}\label{lem:Hdetect}
Suppose that $K\geq 2$. 

\begin{enumerate}
    \item If $k_1,k_2,\ell_1,\ell_2\geq 1$ with $k_1+\ell_1=k_2+\ell_2=K\geq 4$, then 
\[
\mathscr{H}_{K,\left(\begin{smallmatrix}k_1,\ell_1,\chi_1\\ k_2,\ell_2,\chi_2\end{smallmatrix}\right),n}\in \widetilde{\Omega}_{n,N}(\Gamma_1(N)).
\]
\item If $\ell_1+k_1=K\leq 3$, then 
\[
\mathscr{H}_{K,\left(\begin{smallmatrix}k_1,\ell_1,\chi_1\end{smallmatrix}\right),n}\in \widetilde{\Omega}_{n,N}(\Gamma_1(N)).
\]
\end{enumerate}
\end{lem}

We next show a relation between different $D^{\ell}E_{k,\chi,\psi}|S_{N,n}$.
\begin{lem}\label{lem:Echipsi}
\noindent

\noindent
\begin{enumerate}
    \item For $n\in A_N$, $\ell\geq 0$, and $k\in\N$, we have 
\[
D^{\ell}E_{k,\chi,\psi}|S_{N,n}=\chi(n)D^{\ell}E_{k,1,\overline{\chi}\psi}.
\]
\item 
For any $\chi,\psi$ with $\chi\overline{\psi}=\varphi$ and $k+\ell\geq 4$, we have 
    \begin{equation}\label{eqn:Hrewrite}
    H_{k, \ell, \varphi,n}|S_{N,n} =
    \begin{cases}
    \left(\overline{\chi(n)}D^{\ell-1}E_{k, \psi, \chi} - \psi(n) E_{\ell, \overline{\psi}, \overline{\chi}}\right)\Big|S_{N,n}&\text{if }k\equiv \ell\pmod{2},\\
      \left(\overline{\chi(n)}D^{\ell-1}E_{k, \psi, \chi} - \psi(n) D(E_{\ell-1, \overline{\psi}, \overline{\chi}})\right)\Big|S_{N,n}& \text{if }2\leq \ell\not\equiv k\pmod{2},\\
\left(\overline{\chi(n)}E_{k,\psi,\chi}-\overline{\chi(n)}E_{2,\psi,\chi}\right)\Big|S_{N,n}&\text{if $k>2$ even, }\ell=1.
\end{cases}
\end{equation}
\end{enumerate}
\end{lem}
\begin{rem}
Although we assume throughout that the characters $\chi$ and $\psi$ are primitive in the Eisenstein series $E_{k,\chi,\psi}$, we abuse notation here and throughout to write $E_{k,1,\overline{\chi}\psi}$ for $E_{k,1,\varphi}$ if $\overline{\chi}\psi$ is induced from the primitive character $\varphi$.
\end{rem}
\begin{proof}
(1) Since $\gcd(n,N)=1$ and $\chi$ is a character modulo a divisor of $N$, for any $r\in\N$ with $r\equiv n\pmod{N}$ we have 
\[
\sigma_{k-1}^{\chi,\psi}(r) = \sum_{d\mid r} \chi\left(\frac{r}{d}\right)\psi(d) d^{k-1} = \chi(n)\sum_{d\mid r} \left(\overline{\chi}\psi\right)(d) d^{k-1}=\chi(n)\sigma_{k-1}^{1,\overline{\chi}\psi}(r).
\]
(2) For $k+\ell\geq 4$ even, we use part (1) to compute 
\begin{multline*}
\overline{\chi(n)}D^{\ell-1}E_{k, \psi, \chi}|S_{N,n} - \psi(n) E_{\ell, \overline{\psi}, \overline{\chi}}|S_{N,n}=\overline{\chi(n)}\psi(n)D^{\ell-1}E_{k,1,\chi\overline{\psi}}|S_{N,n}-E_{\ell,1,\overline{\chi}\psi}|S_{N,n}\\
=\overline{\varphi(n)}D^{\ell-1}E_{k,1,\varphi}|S_{N,n}-E_{\ell,1,\varphi}|S_{N,n}=H_{k,\ell,\varphi,n}|S_{N,n},
\end{multline*}
which establishes \eqref{eqn:Hrewrite} for $k\equiv \ell\pmod{2}$. For $k\not\equiv \ell\pmod{2}$ and $\ell\geq 2$, we simply note that $H_{k,\ell,\varphi,n}=D(H_{k,\ell-1,\varphi,n})$ by definition. To establish the remaining case in \eqref{eqn:Hrewrite}, for $k'\in\{1,k\}$ we use part (1) to rewrite 
\[
\overline{\chi(n)}E_{k',\psi,\chi}=\overline{\chi(n)\overline{\psi(n)}}E_{k',1,\chi\overline{\psi}}.
\]
\end{proof}

Before proving Theorem \ref{thm:Basis}, we determine a spanning set of the linear combinations of $D^{\ell}E_{k,\chi,\psi}$ with small $k$ and $\ell$.
\begin{prop}\label{prop:SmallK}
Let $N\in\N$ and $n\in \Z$ with $\gcd(N,n)=1$ be given. Then 
\[
\bigg(\bigoplus_{\substack{1\leq k\leq 2\\ 0\leq \ell\leq 1\\ 
\ell+k\leq 2}} D^{\ell}E_{k,\chi,\psi}\bigg)\cap \widetilde{\Omega}_{n,N}(\Gamma_1(N))|S_{N,n}
\]
is spanned by 
\[
\left\{\mathscr{H}_{K,(k,K-k,\chi),n}|S_{N,n}:2\leq K\leq 3\right\}.
\]
\end{prop}
\begin{proof}
Suppose that $f=f_1+f_2$ is in the set, with $f_j$ having $\ell+k=j$. Then, by Lemma \ref{lem:Echipsi} (1), we can write 
\begin{align*}
f_1|S_{N,n}&=\sum_{\chi\text{  odd}} c_{\chi,1}E_{1,1,\chi}S|_{N,n},\\
f_2|S_{N,n}&=\sum_{\chi\text{ even}} c_{\chi,2}E_{2,1,\psi}|S_{N,n}+\sum_{\chi\text{  odd}} c_{\chi,2}DE_{1,1,\chi}|S_{N,n}.
\end{align*}
We show the claim by induction on the number of  pairs of $(\chi,j)$, with $c_{\chi,j}\neq 0$.

If no such pairs exist, then $f|S_{N,n}=0$ and we are done. If there exists an odd character $\chi$
with $\chi(n)=-1$ for which $c_{\chi,j}\neq 0$, then Lemma \ref{lem:Hsingle} implies that the term
\[
c_{\chi,j}D^{j-1}E_{1,1,\chi}=c_{\chi,j}H_{1,j+1,\chi,n}=c_{\chi,j}\mathscr{H}_{j+1,(1,j,\chi),n}
\]
can be removed, and the claim then follows by induction. So we can assume that $c_{\chi,j}=0$ if $\chi(n)=-1$. If there exists an odd $\chi$ with $\chi(n)\neq -1$ and $c_{\chi,j}\neq 0$, then writing
\begin{align*}
c_{\chi,j}D^{j-1}E_{1,1,\chi}&=c_{\chi,j}H_{1,j,\chi,n}+c_{\chi,j}\frac{1+\chi(n)}{1+\varphi_n(n)}D^{j-1}E_{1,1,\varphi_n}\\
&=\mathscr{H}_{j+1,(1,j,\chi),n}+c_{\chi,j}\frac{1+\chi(n)}{1+\varphi_n(n)}D^{j-1}E_{1,1,\varphi_n}
\end{align*}
allows us to replace $\chi$ with $\varphi_n$. if $c_{\varphi_n,j}\neq 0$ initially, then we have reduced the number of non-zero terms and we are done. We can therefore assume that $c_{\chi,j}=0$ for all odd $\chi\neq \varphi_n$. 

Similarly, if $\chi$ is an even character, then we can replace $E_{2,1,\chi}$ with $E_{2,1,1}$ (and possibly adding a multiple of $E_{1,1,\varphi_n}$) by rewriting $\mathscr{H}_{3,(2,1,\chi),n}=H_{2,1,\chi,n}$ in terms of $E_{2,1,1}$ and $E_{1,1,\varphi_n}$. 

We can therefore assume that 
\[
f=c_1 E_{1,1,\varphi_n} + c_2 E_{2,1,1} + c_3 E_{1,1,\varphi_n}.
\]
Then \eqref{eqn:weight1} implies that, for all $p\equiv n\pmod{N}$, we have 
\[
0=c_f(p)=c_1(1+\varphi_n(m))+c_2(p+1) + c_3p(1+\varphi_n)
\]
If $\varphi_n(n)=-1$, then this can only occur for $c_2=0$, so we may assume that $\varphi_n(n)\neq -1$. Taking $p\to\infty$, we conclude that $c_3=-\frac{c_2}{1+\varphi_n}$ and then we compare the constant term to solve for $c_1=-\frac{c_2}{1+\varphi_n(n)}$. So $f$ is a multiple of $\mathscr{H}_{3,(2,1,1),n}$ and the proof is complete.
\end{proof}

Now we move to the proof of Theorem \ref{thm:Basis}. 
\begin{proof}[Proof of Theorem \ref{thm:Basis}]
Suppose that $f\in\widetilde{\Omega}_{n,N}(\Gamma_1(N))$. From Theorem \ref{thm:VanishPrimesArithmetic}, $f|S_{N,n}$ is spanned by elements of $\mathcal{E}(\Gamma_1(N))|S_{N,n}$. From \cite[Proposition 20]{Zagier}, $f|S_{N,n}$ is a linear combination of the derivatives of the level $N$ Eisenstein series and $E_2$ (in Zagier's notation, $\phi=E_2$). Bases for Eisenstein series of different weights are given in \cite[Theorems 4.5.2, 4.6.2, 4.8.1]{DiamondShurman}. Note that an Eisenstein oldform $E_{k,\chi,\psi}(t\tau)$ with $1<t|N$ is annihilated by $S_{N,n}$ for $n\in \mathcal{V}_{m,M}$ (as $(n,N)=1$).   
Suppose that 
\[
f|S_{N,n} = \sum_{i=1}^{t} \alpha_i D^{\ell_i} E_{k_i, \chi_i, \psi_i}|S_{N,n}
\]
for some primitive Dirichlet characters $\chi_i, \psi_i$ modulo a divisor of $N$ that satisfy $\chi_i\psi_i(-1)=(-1)^{k_i}$, where we recall that, for ease of notation, we write $-\frac{1}{12}E_2=E_{2,1,1}$. Due to the symmetry for $k=1$ (since $\chi_i$ and $\psi_i$ is taken as an unordered pair in this case) we may assume without loss of generality that if $\alpha_i D^{\ell_i}  E_{1,\chi_i,\psi_i}$ occurs in the sum, then another term is of the form $\alpha_i D^{\ell_i}E_{1,\psi_i,\chi_i}$. For ease of notation, we shall simply denote $E_{k_i, \chi_i, \psi_i}|S_{N,n}$ by $F_i$.


With $f$ defined as above, let $K_f$ denote the largest value of $k_{i} + \ell_{i}$. We prove the claim by induction on $K_f$. If $K_f\leq 2$, then the claim follows by Proposition \ref{prop:SmallK}, giving the base cases, so we may assume that $K_{f}\geq 3$. 

Suppose without loss of generality that $i=1,\ldots, r$ are the indices for which $k_i+\ell_i=K_f$ and $1\leq i\leq r_1$ have $\ell_i\geq 1$, while $r_1+1\leq i\leq r$ are the indices with $k_i=K_f$ (and $\ell_i=0$). Then 
\[
f_K:=\sum_{i=r_1+1}^{r} \alpha_iF_i
\]
is a pure weight $k$ form without any derivatives. We write
\[
f|S_{N,n} = \sum_{i=1}^{r_1} \alpha_i D^{\ell_i} F_{i} + f_K|S_{N,n}+ g|S_{N,n}
\]
for some quasi-modular form $g$. If we define $K_g$ analogous to $K_f$ above, then we observe that $K_g < K_f$. Let us denote $K_{f}$ as $K$ for simplicity. For a prime $p\equiv n\pmod{N}$, we have 
\begin{multline}\label{eqn:af(p)}
    c_f(p) = 2\sum_{i=1}^{r} \alpha_i p^{\ell_i} \left( \chi_i(p) p^{k_i-1} + \psi_i(p)\right) + c_g(p)\\
    = 2p^{K-1} \sum_{i=1}^{r}\alpha_i\chi_i(n) + \sum_{i=1}^{r} \alpha_i\psi_{i}(n) p^{\ell_i} + c_g(p).
\end{multline}
In particular we have 
\[
c_f(p) = 2p^{K-1} \sum_{i=1}^{r} \alpha_i\left(\chi_i(n)+\delta_{k_i=1} \psi_i(n)\right) + O\left(p^{K-2}\right).
\]
Since $c_f(p)=0$, taking $p\to\infty$ forces
\[
\sum_{i=1}^{r} \alpha_i\left(\chi_i(n) +\delta_{k_i=1}\psi_i(n)\right)= 0.
\]
Letting $\alpha_i'=\alpha_i$ if $k_i\neq 1$ and $\alpha_i'=\frac{1}{2}\alpha_i$ if $k_i=1$, by the assumption above that each unordered pair appears twice for $k_i=1$, we have 
\[
\sum_{i=1}^{r} \alpha_i'\chi_i(n)= 0.
\]
Let $\{e_i\}$ denote the standard basis of $\mathbb{C}^r$. The orthogonal complement of \[
\left(\chi_1(n),\chi_2(n),\ldots, \chi_r(n)\right)
\]
in $\mathbb{C}^r$ is spanned by $\{v_j := \overline{\chi_1(n)} e_1 - \overline{\chi_j(n)} e_j\}$ for $2\leqslant j \leqslant r$. Hence there exist constants $\beta_j$'s such that $(\alpha_1',\ldots, \alpha_r') = \beta_2 v_2 + \ldots + \beta_rv_r$. Without loss of generality, we assume that $k_1=1$ if $k_i=1$ for some $1\leq i\leq r_1$.

Therefore, on rewriting the above equation, we have 
\[
f = \sum_{j=2}^{r} \beta_j \left(\overline{\chi_1(n)}D^{\ell_1}F_1 - \overline{\chi_j(n)} D^{\ell_j} F_j \right) + g.
\] 
For $K\geq 3$ odd, we use \eqref{eqn:Hrewrite} to define a quasi-modular form $h$ such that 
\begin{align*}
    f|S_{N,n} &= \sum_{j=2}^{r} \beta_j \left(\overline{\chi_1(n)}D^{\ell_1}F_1 - \psi_1(n) E_{\ell_1+1, \overline{\chi_1}, \overline{\psi_1}} + \psi_j(n) E_{\ell_j+1, \overline{\chi_j}, \overline{\psi_j}} - \overline{\chi_j(n)} D^{\ell_j} F_j \right) + h|S_{N,n}\\
    &=\sum_{j=2}^{r} \beta_j\mathscr{H}_{K+1,\left(\begin{smallmatrix}k_1,\ell_1+1,\chi_1\overline{\psi_1}\\ k_j,\ell_j+1,\chi_j\overline{\psi_j}\end{smallmatrix}\right),n} \Big|S_{N,n} + h|S_{N,n}.
\end{align*}
By Lemma \ref{lem:Hdetect}, the first sum is contained in $\widetilde{\Omega}_{n,N}(\Gamma_1(N))$ and hence $h\in \widetilde{\Omega}_{n,N}(\Gamma_1(N))$. Moreover, we have 
\[
h=\sum_{\substack{2\leq j\leq r\\ k_j=1}}\beta_j\left(\psi_j(n)E_{\ell_j+1,\overline{\chi_j},\overline{\psi_j}}-\psi_{1}(n)E_{\ell_1+1,\overline{\chi_1},\overline{\psi_1}}\right)\Big|S_{N,n} +h' 
\]
with $K_{h'} < K_f$.  The sum 
\[
\mathcal{H}:=\sum_{\substack{2\leq j\leq r\\ k_j=1}}\beta_j\left(\psi_j(n)E_{\ell_j+1,\overline{\chi_j},\overline{\psi_j}}-\psi_{1}(n)E_{\ell_1+1,\overline{\chi_1},\overline{\psi_1}}\right)
\]
is empty if $k_1\neq 1$, and otherwise $\ell_1=\ell_j=K-1$ for every term in the sum. Therefore either $\mathcal{H}=0$ or (using Lemma \ref{lem:Echipsi} (2))
\begin{align*}
\mathcal{H}|S_{N,n}&=\sum_{\substack{2\leq j\leq r\\ k_j=1}}\beta_j\left(\overline{\chi_j(n)}\psi_j(n)E_{K,1,\chi_j\overline{\psi_j}}-\overline{\chi_1(n)}\psi_{1}(n)E_{K,1,\chi_1\overline{\psi_1}}\right)\Big|S_{N,n}\\
&=\sum_{\substack{2\leq j\leq r\\ k_j=1}}\beta_j\left(\mathscr{H}_{K+1,\left(\begin{smallmatrix}K,1,\chi_j\overline{\psi_j}\\
K,1,\chi_1,\overline{\psi_1}
\end{smallmatrix}\right),n}+E_{1,1,\chi_j\overline{\psi_j}}-E_{1,1,\chi_1\overline{\psi_1}}\right)\bigg|
S_{N,n}.
\end{align*}
Thus,  using Lemma \ref{lem:Hdetect}, we have 
\[
h=\sum_{\substack{2\leq j\leq r\\ k_j=1}}\beta_j\mathscr{H}_{K+1,\left(\begin{smallmatrix}K,1,\chi_j\overline{\psi_j}\\
K,1,\chi_1,\overline{\psi_1}\end{smallmatrix}\right),n}+h''
\]
for some $h''$ with $K_{h''}<K_f$ and $h''\in \widetilde{\Omega}_{n,N}(\Gamma_1(N))$. The claim follows by induction in this case. 

Now suppose that $K\geq 4$ is even. In this case, we use \eqref{eqn:Hrewrite} to define a quasi-modular form $h$ such that 
\[
    f|S_{N,n} =\sum_{j=2}^{r} \beta_j\mathscr{H}_{K+1,\left(\begin{smallmatrix}k_1,\ell_1+1,\chi_1\overline{\psi_1}\\ k_j,\ell_j+1,\chi_j\overline{\psi_j}\end{smallmatrix}\right),n} \Big|S_{N,n} + h|S_{N,n}.
\]
By Lemma \ref{lem:Hdetect}, the first sum is contained in $\widetilde{\Omega}_{n,N}(\Gamma_1(N))$ and hence $h\in \widetilde{\Omega}_{n,N}(\Gamma_1(N))$. Moreover, we have (note that if $k_j=1$ then $\ell_j=K-1\geq 3$ must be odd)
\[
h=\sum_{\substack{2\leq j\leq r\\ k_j=1}}\beta_j\left(\psi_j(n)DE_{\ell_j,\overline{\chi_j},\overline{\psi_j}}-\psi_{1}(n)DE_{\ell_1,\overline{\chi_1},\overline{\psi_1}}\right)\Big|S_{N,n} +h' 
\]
with $K_{h'} < K_f$.  The sum 
\[
\mathcal{H}:=\sum_{\substack{2\leq j\leq r\\ k_j=1}}\beta_j\left(\psi_j(n)DE_{\ell_j,\overline{\chi_j},\overline{\psi_j}}-\psi_{1}(n)DE_{\ell_1,\overline{\chi_1},\overline{\psi_1}}\right)
\]
is empty if $k_1\neq 1$, and otherwise $\ell_1=\ell_j=K-1$ for every term in the sum. Therefore either $\mathcal{H}=0$ or (using Lemma \ref{lem:Echipsi} (2))
\begin{align*}
\mathcal{H}|S_{N,n}&=\sum_{\substack{2\leq j\leq r\\ k_j=1}}\beta_j\left(\overline{\chi_j(n)}\psi_j(n) DE_{K-1,1,\chi_j\overline{\psi_j}}-\overline{\chi_1(n)}\psi_{1}(n)DE_{K-1,1,\chi_1\overline{\psi_1}}\right)\bigg|S_{N,n}\\
&=\sum_{\substack{2\leq j\leq r\\ k_j=1}}\beta_j\left(\mathscr{H}_{K+1,\left(\begin{smallmatrix}K-1,2,\chi_j\overline{\psi_j}\\
K-1,2,\chi_1,\overline{\psi_1}
\end{smallmatrix}\right),n}+DE_{1,1,\chi_j\overline{\psi_j}}-DE_{1,1,\chi_1\overline{\psi_1}}\right)\bigg|
S_{N,n}.
\end{align*}
Thus,  using Lemma \ref{lem:Hdetect}, we have 
\[
h=\sum_{\substack{2\leq j\leq r\\ k_j=1}}\beta_j\mathscr{H}_{K+1,\left(\begin{smallmatrix}K-1,2,\chi_j\overline{\psi_j}\\
K-1,2,\chi_1,\overline{\psi_1}\end{smallmatrix}\right),n}+h''
\]
for some $h''$ with $K_{h''}<K_f$ and $h''\in \widetilde{\Omega}_{n,N}(\Gamma_1(N))$. The claim again follows by induction in this case. 
\end{proof}

    \section{Finite checks for prime detection}\label{sec:FiniteCheckAP}
We first prove that one only needs to check finitely-many primes in arithmetic progressions to verify that a form is prime-detecting in that arithmetic progression.
 \begin{proof}[Proof of Theorem \ref{thm:FiniteCheckAP}]   The proof is a ready adaptation of the proof of \cite[Theorem 1.3]{KKL1} Suppose that 
    \[
    f = \sum_{\ell,k} \alpha_{\ell,k} D^{\ell} E_{k,\chi,\psi}|V_{\delta}.
    \]
    Choose $R$ to be the maximum of $\ell + k-1$ for which $\alpha_{\ell,k}\neq 0$. 
Since the cusidal part of $f$ vanishes by assumption, Theorem \ref{thm:VanishPrimesArithmetic} (1) implies that $f\in\widetilde{\Omega}_{m,M}(\Gamma_1(N))$, if and only if $f\in\widetilde{\Omega}_{n,N}(\Gamma_1(N))$ for then every $n\in\mathcal{V}_{m,M}$. From the proof of Lemma \ref{lem:EisensteinArithmetic}, we may write 
    \[
    c_f(p)=\sum_{r=0}^{R} \beta_r(\chi(n),\psi(n)) p^r,
    \]
    whenever $p \equiv n\mod N$ is a prime. Now, if $c_f(p_i) = 0$ for $1\leqslant i\leqslant R+1$, then we obtain the system of equations
    \[
        \begin{pmatrix}
            1&p_1&p_1^2&\ldots&p_1^R\\
            1&p_2&p_2^2&\ldots&p_2^R\\
            \vdots&\vdots&\vdots&\ddots&\vdots\\
            1&p_{R+1}&p_{R+1}^2&\ldots&p_{R+1}^R
        \end{pmatrix}
        \begin{pmatrix}
            \beta_0(\chi(n),\psi(n)\\ \beta_1(\chi(n),\psi(n))\\ \vdots \\ \beta_R(\chi(n),\psi(n))
        \end{pmatrix} = 0.
    \]
    This system, being a Vandermonde system, is uniquely solvable and hence $\beta_0(\chi(n),\psi(n)) = \beta_1(\chi(n),\psi(n)) = \ldots = \beta_R(\chi(n),\psi(n)) = 0$. Repeating this for $R+1$ primes in each of the arithmetic progressions $n$ modulo $N$ for $n\in\mathcal{V}_{m,M}$ yields the claim.
\end{proof}
Before proving Corollary \ref{cor:N=1}, we require the following lemma.
\begin{lem}\label{lem:SieveVanish}
If $g\in \mathcal{S}$ is a level one quasi-modular cusp form and $g|S_{M,m}=0$ for some $M\in\N$ and $m\in\Z$, then $g=0$.
\end{lem}
\begin{proof}
Let $g$ be a quasi-modular cusp form of level $1$ and mixed weight $\leq 2k$. By \cite[(3.1)]{BvIMOS}, we see that,
$$
g\in \bigoplus_{j=1}^{k}\bigoplus_{{\rm r}=0}^{j-1} D^rS_{2j-2{\rm r}}
$$
where $S_{2k}$ is the space of cusp forms of level 1 and weight $2k$, and $D:=q\frac{d}{dq}$ is the differential operator. As in the proof of Theorem~1.1 in \cite{BvIMOS}, we can then decompose $g$ into a sum  of (derivatives of) normalized cuspidal Hecke eigenforms of various weights, and construct a compositum number field $K$ from the Fourier coefficients of these eigenforms. Consequently, 
we can assume 
\begin{align}\label{eq1}
g= \sum_{r=2}^{2k} \sum_j h_{r,j}(D)f_{r,j}
\end{align}
where $f_{r,j}$ is a normalized cuspidal Hecke eigenform of level 1 and weight $r$, and  $h_{r,j}\in \mathcal{O}_K[x]$ is a nonzero polynomial of degree at most $k$. Here $\mathcal{O}_K$ denotes the ring of integers of $K$. 

Our task is to argue, as in \cite{BvIMOS}, that all $h_{r,j}$ are zero polynomials if $g|S_{m,M}=0$. From \eqref{eq1}, 
$$
a_g(p) = \sum_{r=2}^k \sum_j h_{r,j}(p) a_{f_{r,j}}(p)
$$
at every prime $p$. Without loss of generality, we assume all $h_{r,j}(m)\neq 0$ (which can be achieved by replacing $m$ with a sufficiently large term in the arithmetic progression $m+M\Z$). Thus all $h_{r,j}(m)^{-1}\in K$. Fix a pair $(r',j')$.

Now we apply \cite[Lemma~2.1]{BvIMOS} as follows: choose a sufficiently large prime $\ell$ that fulfills (B1)-(B4) as required in the lemma, and take $d=M\ell$. The lemma assures that there is a positive portion of primes $p\equiv m$ mod $M\ell$ for which 
$$
a_{f_{r,j}}(p)\equiv h_{r',j'}(m)^{-1}\delta_{(r,j)=(r',j')}\quad  \pmod{\mathfrak{p}_{\ell,K}}, \quad \forall \ (r,j),
$$
where $\delta_{a=b}=1$ if $a=b$ or $0$ otherwise, and $\mathfrak{p}_{\ell,K}\subset \mathcal{O}_K$ is a prime ideal lying above $\ell$. For such primes $p$ (which satisfy $p\equiv m$ mod $M\ell$), we have $h_{r,j}(p)\equiv  h_{r,j}(m)$ mod $\ell$ and 
$$
a_g(p) = \sum_{r=2}^k \sum_j h_{r,j}(p) a_{f_{r,j}}(p)\equiv h_{r',j'}(m)h_{r',j'}(m)^{-1}=1 \quad \pmod{\mathfrak{p}_{\ell,K}}.
$$
However, if $g|S_{m,M}=0$, then $a_g(p)=0$ for all $p\equiv m$ mod $M$. Contradiction arises, so all $h_{r,j}$ are zero polynomials.
\end{proof}
We are now ready to prove Corollary \ref{cor:N=1}.
\begin{proof}[Proof of Corollary \ref{cor:N=1}]
Suppose that $f\in\widetilde{\Omega}_{m,M}(\SL_2(\Z))$. We split $f=E+g$ with $
E\in \mathcal{E}$ and $g\in\mathcal{S}$. Since $\gcd(n,1)=1$ is satisfied for all $n$, Theorem \ref{thm:VanishPrimesArithmetic} implies that $f|S_{M,m}\in \mathcal{E}|S_{M,m}$. In other words, $f|S_{M,m}=E|S_{M,m}$ and $g|S_{M,m}=0$. By Lemma \ref{lem:SieveVanish}, we have $g=0$, and hence $f=E$.

We claim next that $E$ detects all primes. To see this note that, for $N=1$, the Vandermonde system built in the proof of Theorem \ref{thm:FiniteCheckAP} is the same as the Vandermonde system built in the proof of \cite[Theorem 1.3]{KKL1}. Since $E|S_{M,m}$ vanishes at all primes by Theorem \ref{thm:FiniteCheckAP}, it satisfies this Vandermonde system, so \cite[Theorem 1.3]{KKL1} implies that $E$ vanishes at all primes. We conclude that $g=0$ and hence $f=E\in \widetilde{\Omega}$.

Moreover, since $f=E$ and $c_E(p)=0$ for all primes (and hence any primes $p\not\equiv m\pmod{M}$), we see that $f$ cannot strongly detect primes in the arithmetic progression unless $M=1$.
\end{proof}

\appendix

\section{An analytic proof for the level one case}\label{sec:AnalyticProof}

As in \cite{CvIO}, polynomial expressions involving MacMahon partition functions can be rewritten as polynomial equations involving various divisor functions. Thus the study of partition identities that detect primes is in principle a study of ``divisor function'' identities that detect primes. In this spirit we consider the following general situation.

For $1\leqslant i \leqslant r$, we choose polynomials $P_i(x)\in \Q[x]$. We also choose and fix non-negative integers $\{k_i\}_{i=1}^{r}$. We define the function
	\begin{equation}\label{Equation "a(n) definition"}
		a(n) := \sum_{i=1}^{r} P_i(n) \sigma_{k_i}(n) = \sum_{j=1}^{t} A_j n^{\ell_j} \sigma_{k_j}(n),
	\end{equation}
	for some $A_j\in \Q$ and (not necessarily distinct) non-negative integers $\ell_j$. We observe that 
	\begin{equation}\label{Equation "W definition"}
		W(s):= \sum_{j=1}^{t} A_j \zeta(s-\ell_j)\zeta(s-\ell_{j}-k_{j}) = \sum_{n=1}^{\infty} \frac{a(n)}{n^s}.
	\end{equation}
    Associate to $W$, two integers $R_{W}$ and $S_W$ defined by $R_{W} := \max_{j} \{\ell_j, \ell_j + k_j\} = \max_j \{\ell_j+k_j\}$ and $S_{W}:=\sum |A_j|$, where the sum is over all indices $j$ such that $\ell_j + k_j = R_W$.	We also define the closely related function
	\begin{equation}
		Z_{W}(s) := \prod_{j=1}^{t} \left(\zeta(s-\ell_{j}) \zeta(s-\ell_{j}-k_{j}) \right)^{A_j} = \sum_{n=1}^{\infty} \frac{b(n)}{n^s}
	\end{equation}
	for $\Re(s) \gg 1$. In order to state the main theorem, we introduce the following notation. For a quadruple of integers $\textbf{m}=(m_1,m_2,m_3,m_4)\in \Z^4$, we define
	\begin{multline}\label{Equation "W_m definition"}
		W_{\textbf{m}}(s):= \zeta(s-m_1)\zeta(s-m_3) + \zeta(s-m_2)\zeta(s-m_4)\\
		-\zeta(s-m_1)\zeta(s-m_4) - \zeta(s-m_2)\zeta(s-m_3) =: \sum_{n=1}^{\infty} \frac{a_{\textbf{m}}(n)}{n^s}.
	\end{multline}

	\begin{thm}\label{thm:Equivalence}
		Let notation be as above and fix a $W$ as in \eqref{Equation "W definition"}. Then the following are equivalent.
		\begin{enumerate}
			\item	$a(p)$ vanishes for all primes $p$,
			\item 	$Z_{W}(s)\equiv 1$,
			\item 	There exist integers $\{c_{\textbf{m}}\}_{\textbf{m}\in \Z^4}$, at most finitely many of them non-zero, such that $W(s) = \sum_{\textbf{m}\in \Z^{4}} c_{\textbf{m}}W_{\textbf{m}}(s)$,
            \item   There exist $R_W$ distinct primes $\{p_1, \ldots, p_{R_W}\}$ such that $a(p_i)=0$ for $1\leqslant i\leqslant R_W$.
		\end{enumerate}
	\end{thm}

\begin{proof}

%

    The equivalence of (1) and (4) is the content of \cite[Theorem 1.3]{KKL1}. So we shall prove that (1), (2) and (3) are equivalent to one another.
    
		Let us prove that (1) implies (2). Clearing out denominators in the definition of $a(n)$ and by taking a suitable power of $Z_{W}(s)$ we may without loss of generality assume that all the $A_j$'s are integers.
        After some rearrangement, we may write
		\[
		Z_{W}(s) = \prod_{j=1}^{u} \zeta(s-m_j)^{B_j}
		\]
		where, $m_1 < m_2 < \ldots < m_u$. Let $M:=\{m_j: 1\leqslant j \leqslant u\}$. To prove (2), it suffices to show that $B_j=0$ for all $j=1,\ldots,u$. From definition, observe that $b(n)$ is a multiplicative function. In other words, $Z_{W}$ has an Euler product expansion. We have
        \begin{align*}
        \log Z_W(s) 
        & = \sum_p b(p)p^{-s} + \cdots\\
        & = \sum_p \Big(\sum_{m_j\in M} B_j p^{m_j}\Big) p^{-s} +\cdots.  
        \end{align*}
		On the other hand, we observe that $a(p) = b(p)$ for every prime $p$. This implies that $B_j=0$ for every $1\leqslant j\leqslant u$. This proves (2).
		
		Let us assume (2) and prove (3). Henceforth we shall adopt the notation of \eqref{Equation "W definition"}. In that notation, we shall suppose that $A_j\neq 0$ for every $1\leqslant j \leqslant t$. For simplicity, we shall also suppose that $(\ell_{i}, k_i) \neq (\ell_{j}, k_{j})$ if $i\neq j$. We shall proceed by successive reductions, first on $S_W$ and then on $R_W$, and so on. If $W\equiv 0$, then we may choose $c_{\textbf{m}}=0$ for all $\textbf{m}$ and we are done. So we suppose otherwise, that is $W\neq 0$. In particular, since $Z_{W}\equiv 1$ by assumption, we have $S_{W} \geqslant 2$ and, since all $\ell_j$ and $k_j$ are non-negative, $R_W\ge 0$.
		
		Without loss of generality, suppose that $\ell_t + k_t = R_{W}$. Define $m_4 = \ell_{t}+k_{t}$. Suppose also without loss of generality that $A_t > 0$.  Since $Z_{W}(s)\equiv 1$, there exists at least one $j < t$ such that $\ell_j+k_j = \ell_t + k_t$ and $A_j < 0$. Since we have chosen the tuples $(\ell_i, k_i)$'s to be distinct it follows that $\ell_{j}\neq \ell_{t}$ and $k_{j}\neq k_{t}$. Relabeling indices if necessary we may suppose that $k_{t} > k_{j}$. In particular $k_{t}\neq 0$. Choose and fix such a $j$. We set $m_2 = \ell_{j}$ and $m_1 = m_3 = \ell_t$. Define $\tilde{W}(s) := W(s)+W_{\textbf{m}}(s)$. We observe that $Z_{\tilde{W}}(s) = Z_{W}(s) \equiv 1$ by assumption. We crucially observe that $R_{\tilde{W}}\leqslant R_{W}$ and if $R_{\tilde{W}} = R_{W}$, then $S_{\tilde{W}} \leqslant S_{W} - 2 < S_{W}$. Furthermore, $R_{\tilde{W}} < R_{W}$ if $S_{W}=2$. If $\tilde{W}\neq 0$, repeating the above process for $\tilde{W}$, we may ``peel off" one $W_{\textbf{m}}$ at a time from $W$.        
        
        To complete the proof, we need to make sure that this process terminates in finitely many steps. To see this, first observe that $R_{W}$ is non-increasing in this process. Second, at each step, at least one of $R_{W}$ or $S_{W}$ is strictly decreasing. Moreover, whenever $R_{W}$ is non-decreasing, the parameter $S_{W}$ is strictly decreasing, ultimately forcing $R_{W}$ to decrease after finitely many steps. Even though $S_{W}$ grows occasionally\footnote{When $S_{W}=2$, there are exactly two choices for $j$ such that the maximum $R_{W}$ is attained. We cancel them out by adding the corresponding $W_{\textbf{m}}$ and obtain $\tilde{W}$. In this step, $R_{\tilde{W}} < R_{W}$, but $S_{\tilde{W}}$ now counts the sum of coefficients of the pairs $(\ell_{j}, k_{j})$ such that $\ell_{j} + k_{j} = R_{\tilde{W}}$ and not $R_{W}$. Thus $S_{\tilde{W}} > S_{W}$ (and in fact this might be considerably larger).}, it is at most finite, at each step, and hence eventually goes down to zero, which in turn decreases $R_{W}$, keeping the reduction argument going. Continuing this process, we eventually end up with $R_{W}=0$ and $S_{W}= 2$; but then $W = \zeta^2(s) - \zeta^2(s) = 0$ (since $Z_{W}\equiv 1$). Following this procedure, we may write $W$ as a integral linear combination of $W_{\textbf{m}}$'s proving (3).

		Finally let us suppose (3) and prove (1). It suffices to verify that for every $\textbf{m}\in \Z^{4}$, and for every prime $p$, we have $a_{\textbf{m}}(p)=0$. We shall give the proof assuming $m_1 < m_2 < m_3 < m_4$, the other cases being treated similarly. Given $\textbf{m} = (m_1,m_2,m_3,m_4)\in \Z^{4}$, by direct computation, we have that 
		\begin{multline}\label{eqn:FourierCoefficientCalculation}
			a_{\textbf{m}}(p) = p^{m_1} \sigma_{m_3-m_1}(p) + p^{m_2}\sigma_{m_4-m_2}(p)-p^{m_1}\sigma_{m_4-m_1}(p) - p^{m_2}\sigma_{m_3-m_2}(p)\\
			= p^{m_3} +p^{m_1} + p^{m_4} + p^{m_2}-p^{m_4} - p^{m_1} - p^{m_3} - p^{m_2}=0.
		\end{multline}
		This completes the proof of the theorem.
\end{proof}

    \begin{rem}
        It is natural to want to extend this proof to forms of higher level, but this does not seem to be straightforward. When considering Eisenstein series of higher level, the associated Dirichlet series involves products of Dirichlet $L$-functions. More precisely, the analogue of $Z_{W}$ (say $Z_{W,N}$, for level $N$) in this situation is no longer a ratio of shifts of the Riemann zeta function, but of $L$-functions associated to Dirichlet characters. The pole of $\zeta(s)$ at the point $s=1$ was used to pinpoint the rightmost singularity of $\log(Z_{W}(s))$. But, as it is well known that the $L$-function associated to non-principal Dirichlet characters have neither zeros nor poles on the boundary of absolute convergence, we run into trouble when looking for the rightmost singularity of $\log(Z_{W,N}(s))$, unless a principal character appears in the decomposition. Futhermore, GRH predicts that $\log(Z_{W,N}(s))$ should not have any poles in the vertical strip of width $1/2$ to the left of region of convergence if there is no principal character. A workaround to this obstacle seems to require new ideas.
    \end{rem}

Along with the vanishing at the primes, it is interesting to investigate when $a(n)$ is non-negative. For a general $W$ as above, the answer depends on the  of $c_{\textbf{m}}$. For $W_{\textbf{m}}$ however, we have the following precise result.

\begin{lem}
    Let $\textbf{m}=(m_1,m_2,m_3,m_4)$ be given. Then $\sgn(a_{\textbf{m}}(n)) = \sgn(m_2-m_1)\sgn(m_4-m_3)$ for every composite number $n$.
\end{lem}

\begin{proof}
    We need only generalize the calculation in \eqref{eqn:FourierCoefficientCalculation}. Suppose $n$ is a composite number. From definition,
    \begin{align*}
        a_{\textbf{m}}(n)&=n^{m_1} \sigma_{m_3-m_1}(n) + n^{m_2}\sigma_{m_4-m_2}(n)-n^{m_1}\sigma_{m_4-m_1}(n) - n^{m_2}\sigma_{m_3-m_2}(n)\\
        &=\sum_{d|n} (n^{m_1} d^{m_3-m_1} + n^{m_2} d^{m_4-m_2} - n^{m_1} d^{m_4-m_1} - n^{m_2} d^{m_3-m_2})\\
        &=\sum_{d|n} \left(\left(\frac{n}{d}\right)^{m_2} - \left(\frac{n}{d}\right)^{m_1}\right) (d^{m_4}-d^{m_3}).
    \end{align*}
    The terms corresponding to $d=1$ and $d=n$ vanish. The remaining terms all non-zero and have the same sign which is $\sgn(m_2-m_1)\sgn(m_4-m_3)$. The lemma follows from here.
\end{proof}

\bibliographystyle{amsalpha}

\bibliography{Bibliography}

\end{document}